\documentclass[3p,times]{elsarticle}

\usepackage{amsmath} 
\usepackage{amssymb}

\usepackage{amsthm}
\usepackage{amsfonts}
\usepackage{graphicx}
\usepackage{subfigure}
\usepackage{subfigmat}
\usepackage{framed}
\usepackage[usenames,dvipsnames]{color}
\usepackage{natbib}

\usepackage{algorithm, algorithmic}
\numberwithin{algorithm}{section}
\usepackage{enumerate}
\usepackage{url}

\theoremstyle{plain}
\newtheorem{lem}{Lemma}
\newtheorem{thm}{Theorem}
\newtheorem{prop}{Proposition}
\newtheorem{cor}{Corollary}
\theoremstyle{definition}

\newtheorem{assume}{Assumption}

\newcommand{\R}{\mathbb{R}}

\newcommand{\eps}{\epsilon}
\newcommand{\ve}{\varepsilon}

\newcommand{\m}{\mathcal}

\newif\iflecture
\lecturefalse

\begin{document}

\begin{frontmatter}

\title{Distributionally Robust Bootstrap Optimization}
\author[label1]{Tyler Summers
}
\author[label2]{\quad \quad Maryam Kamgarpour}

\affiliation[label1]{organization={Control, Optimization, and Networks Lab, University of Texas at Dallas}}
\affiliation[label2]{organization={SYCAMORE Lab, Ecole Polytechnique Federale de Lausanne (EPFL)}}%


\begin{abstract}
Control architectures and autonomy stacks for complex engineering systems are often divided into layers to decompose a complex problem and solution into distinct, manageable sub-problems. To simplify designs, uncertainties are often ignored across layers, an approach with deep roots in classical notions of separation and certainty equivalence. But to develop robust architectures, especially as interactions between data-driven learning layers and model-based decision-making layers grow more intricate, more sophisticated interfaces between layers are required. We propose a basic architecture that couples a statistical parameter estimation layer with a constrained optimization layer. We show how the layers can be tightly integrated by combining bootstrap resampling with distributionally robust optimization. The approach allows a finite-data out-of-sample safety guarantee and an exact reformulation as a tractable finite-dimensional convex optimization problem. 
\end{abstract}


\end{frontmatter}

\section{Introduction}
Control architectures and autonomy stacks for complex engineering systems are often divided into layers to decompose a complex problem and solution into distinct, manageable sub-problems. Typical layers include state estimation, parameter estimation, system identification, environment perception, prediction, planning/optimization, and feedback control. To simplify designs, uncertainties are often ignored across layers, an approach with deep roots in classical notions of separation and certainty equivalence. However, small uncertainties can become amplified across layers and lead to severe risks. To develop robust architectures, especially as interactions between data-driven learning layers and model-based decision-making layers grow more intricate, much more sophisticated interfaces between layers are required.

We consider an architecture that couples a statistical parameter estimation layer with a constrained optimization layer. The problems within each layer individually are quite standard with long histories and vast literatures. The interface between statistics and decision theory has also been studied since the seminal work of Wald \cite{wald1950statistical}. But there remain important open problems and opportunities at this interface, especially in light of recent developments in distributionally robust optimization. In the statistical estimation layer, we consider a linear regression problem and ordinary least squares estimator, along with a bootstrap resampling method \cite{freedman1981bootstrapping,efron1994introduction} to quantify the unknown distribution of the parameter estimation error. In the decision-making layer, we consider an optimization problem with safety constraints that depend on an \emph{estimator} derived from raw data; for concreteness, we consider estimation of the unknown regression parameter. Since this parameter must be estimated from finite noisy data, any estimate used in the optimization problem comes with a risk of violating the constraint. Robust and stochastic optimization, including the emerging area of distributionally robust optimization \cite{wiesemann2014distributionally,kuhn2019wasserstein}, provide approaches for handling uncertainties in optimization problems. However, there are still crucial gaps between inferential uncertainty representations and uncertainty representations assumed in optimization models. 


Our main contributions are:
\begin{enumerate} 
\item We propose a distributionally robust bootstrap optimization algorithm based on a Wasserstein ambiguity set built directly on the bootstrap samples. This approach tightly integrates an important statistical technique (bootstrapping) with the emerging area of distributionally robust optimization in an end-to-end framework that directly generates robust decisions from data across the architecture layers.
\item We prove that by carefully selecting the radius of the ambiguity set, we obtain a finite-sample safety guarantee that limits the out-of-sample constraint violation. The radius depends on the amount of regression data, number of bootstrap samples, parameter dimension, tail properties of the error distribution, and a confidence parameter.
\item Leveraging recent work, we derive a finite-dimensional convex program that establishes the out-of-sample safety guarantee on risk of constraint violation.
\item We show in a simple numerical example that small parameter estimation errors can lead to arbitrarily large constraint violations and that the proposed approach rigorously limits such risks.
\end{enumerate}

The rest of the paper is structured as follows. Section II presents the architecture that interfaces statistical regression and constrained optimization. Section III derives finite-sample bounds for Wasserstein distances for bootstrapping in regression models. These bounds inform the ambiguity set radius that leads directly to an out-of-sample safety guarantee. Section IV presents an exact reformulation of the optimization problem as a finite-dimensional convex program. Section V provides numerical experiments, and Section VI concludes.

\section{A Data and Decision-Making Interface}
We consider a decision-making architecture where a regression model (Section II.A) is coupled with an optimization problem (Section II.B). 
\begin{figure}
    \centering
    \includegraphics[width=0.99\linewidth]{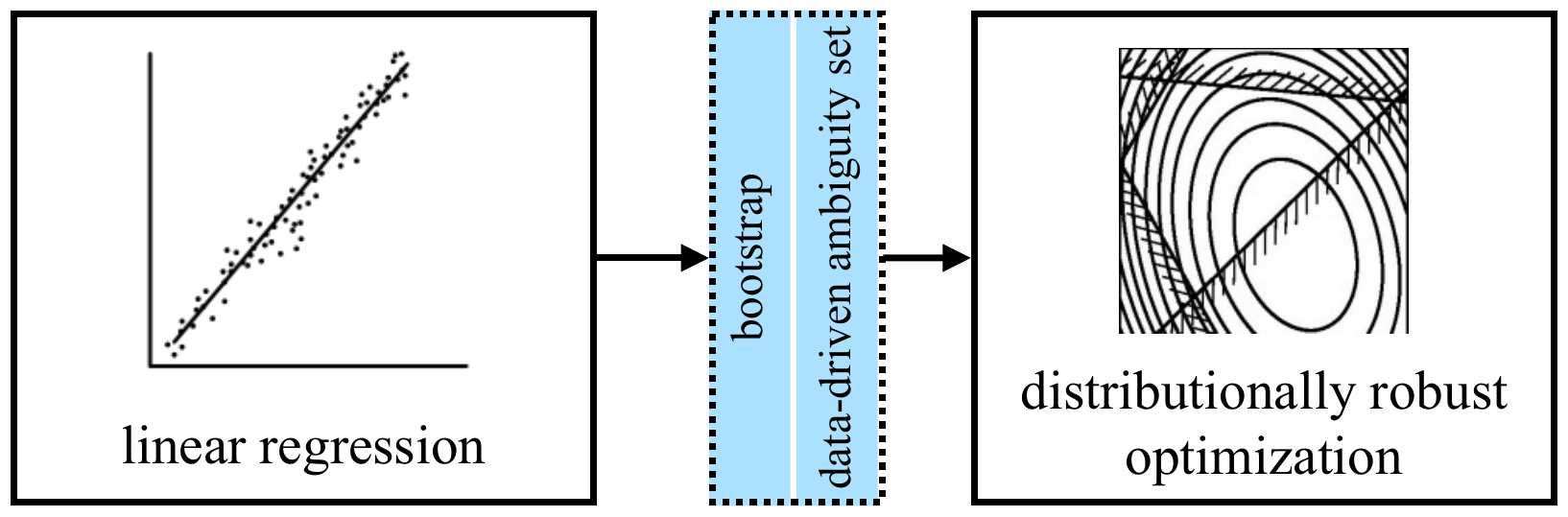}
    \caption{A data-decision architecture with regression and optimization layers. The constraints in the optimization layer depend on an estimator derived from raw regression data with unknown distribution. We integrate the layers by building an ambiguity set for distributionally robust optimization based on bootstrap samples of the regression data.}
    \label{fig:my_label}
\end{figure}

\subsection{Statistical Estimation with Linear Regression} 
We consider a linear regression model
\begin{equation} 
\label{eq:least_squares}
y_n = X_n \beta + \varepsilon_n,
\end{equation}
where $y_n \in \mathbf{R}^n$ is a data vector, $X_n \in \mathbf{R}^{n\times p}$ is a (fixed) data/design matrix, $\beta \in \mathbf{R}^p$ is an unknown parameter vector to be estimated from the data, and $\varepsilon_n \in \mathbf{R}^n$ is a residual vector, with its components denoted by $\ve_{ni} \in \mathbf{R}$. Let $\mathcal{M}(\mathbf{R}^d)$ denote the set of all probability distributions on $\mathbf{R}^d$.

\begin{assume}
\label{assm:noise_dist_asubexp}
The residual components $\varepsilon_{ni}$ are independent and identically distributed (iid) from an unknown distribution function $F \in \mathcal{M}(\mathbf{R})$ having mean zero, variance $\sigma^2$, and \emph{$\alpha$-sub-exponential} tails. In particular, there are constants\footnote{It is possible to obtain all of our results for $\alpha \in (0,2)$, but this case requires some tedious technicalities for the Wasserstein concentration bounds in \cite{fournier2015rate}. We therefore assume $\alpha > 2$ to simplify the exposition.} $\alpha > 2$ and $K>0$ such that $P(|\varepsilon_{ni}| > t) \leq 2\exp(-t^\alpha / K^\alpha)$. Equivalently, there are constants $\alpha > 2$ and $\gamma > 0$ such that $ \mathbf{E} \exp(\gamma | \varepsilon_{ni} |^\alpha) < \infty$. The corresponding (finite) Orlicz norm\footnote{Note that for $\alpha < 1$ this is actually a quasinorm, since the triangle inequality does not hold. Nevertheless, many norm-like properties can be recovered up to $\alpha$-dependent constants.} of order $\alpha$ is denoted by $\| \varepsilon_{ni} \|_{\Psi_\alpha} := \inf \{ t > 0 \mid \mathbf{E} \exp((|\varepsilon_{ni}|/t)^\alpha) < 2 \}$
\end{assume}


The ordinary least squares estimate $\hat \beta_n$ for $\beta$ is given by
\begin{align}
\label{eq:regression}
\hat{\beta}_n = ( X_n^T X_n)^{-1}X_n^T y_n.
\end{align}
Note that $\hat{\beta}_n$ is a random variable due to the residual distribution $F$. This random estimate interfaces with a constrained optimization problem as follows.

\subsection{Optimization with Data-Dependent Safety Constraints} In a decision-making layer, we consider a constrained optimization problem with decision variable $x\in \mathbf{R}^d$ whose objective function is convex and constraint functions\footnote{It's also possible to allow the objective function to depend on $\beta$.} depend affinely on the regression parameter $\beta$:
\begin{equation} 
\label{prob:original}
\begin{aligned} 
 &\text{minimize} \quad  && f(x)   \\ 
 &\text{subject to} && a_j^T(\beta) x \leq b_j(\beta), \quad j = 1,...,m. 
\end{aligned}
\end{equation}
Since $\beta$ is unknown, a widely-used approach to simplify the interface between data and decision is to plug in the nominal point estimate $\hat \beta_n$ and solve the certainty-equivalent problem
\begin{equation} 
\label{prob:uncertain}
\begin{aligned} 
 &\text{minimize} \quad  && f(x)   \\ 
 &\text{subject to} && a_j^T(\hat \beta_n) x \leq b_j(\hat \beta_n), \quad j = 1,...,m.
\end{aligned}
\end{equation}
However, as $\hat \beta_n$ is a random variable computed from noisy finite data, an optimal solution to this problem may be infeasible for the original problem; this approach inherently incurs a risk of constraint violation. Moreover, the lack of knowledge of the exact distribution function of $\hat \beta_n$ makes even evaluating this risk difficult. Our goal is to limit the risk of constraint violation by explicitly accounting for the estimation uncertainty in $\hat \beta_n$.


To incorporate our lack of knowledge of the distribution function itself, we consider a distributionally robust risk constrained optimization problem 
\begin{equation} 
\label{prob:dist_robust}
\begin{aligned} 
 &\text{minimize} \quad  && f(x)   \\ 
 &\text{subject to} && \sup_{P_{\hat \beta_n} \in \m{P} } \rho(a_j^T(\hat \beta_n) x - b_j(\hat \beta_n)) \leq \Delta_j,  \quad j = 1,...,m,
\end{aligned}
\end{equation}
where $\rho:\mathcal{L} \rightarrow \mathbf{R}$ is a risk measure, which is a functional that maps random variables in some space $\mathcal{L}$ to real numbers, $\Delta_j$ is a user-prescribed risk limit, and $\mathcal{P}$ is an ambiguity set of distributions for $\hat \beta_n$. There are several different ways to quantify risk (see, e.g., \cite{artzner1999coherent} for an axiomatic approach involving the notion of risk coherence), with some important examples including expectation, Value at Risk (VaR), and  Conditional Value at Risk (CVaR) \cite{rockafellar2000optimization}. To explicitly connect this optimization problem to the regression problem, the ambiguity set should be built directly upon the finite regression data $(y_n,X_n)$. Thus, we focus on constructing an empirical distribution based on bootstrapping to help us define the ambiguity set.

\subsection{Bootstrapping Regression Models} 
If we knew the residual distribution $F$, we could use the relation in \eqref{eq:regression} to characterize the distribution of $\hat{\beta}_n$. In the lack of knowledge of $F$, bootstrapping is a flexible and powerful approach to infer an empirical distribution for $\hat{\beta}_n$. 
The basic idea of bootstrapping in regression models is to approximate the unknown distribution function $\Phi(F)$ of $\hat \beta_n$ by a bootstrap distribution $\Phi(\hat F_n)$ obtained from resampling the centered estimated residuals. 

Let us define the estimated residuals as
\begin{align*}
\hat \varepsilon_n = y_n - X_n \hat \beta_n. 
\end{align*}
The centered empirical distribution of residuals is denoted by $\hat{F}_n$, which assigns probability $1/n$ to each element $\hat \ve_{ni} - \frac{1}{n} \sum_{i=1}^n \hat \ve_{ni}$, $i=1, \dots, n$.  We can draw samples from $\hat F_n$ by resampling with replacement the elements of the centered residuals. We draw $k$ bootstrap resamples $\{ \varepsilon_n^{*i} \}_{i=1}^{k}$, with components of each $\varepsilon_n^{*i} \in \mathbf{R}^n$ drawn iid from $\hat F_n$, and construct the ``fake'' bootstrap data
\begin{equation}
y_n^{*i} = X_n \hat \beta_n + \varepsilon_n^{*i}, \quad i=1,...,k.
\end{equation}
We then form the bootstrap parameter estimates
\begin{equation}
\hat \beta_n^{*i} = (X_n^T X_n)^{-1} X_n^T y_n^{*i}, \quad i=1,...,k.
\end{equation}
The estimates $\hat \beta_n^{*i}$, $i = 1, \dots, k$, empirically approximate the unknown distribution function  $\Phi(F)$. This bootstrap approximation for regression models has been shown to be valid asymptotically  \cite{bickel1981some,freedman1981bootstrapping}. Here, we are interested in developing non-asymptotic finite sample bounds to define a bootstrap-based ambiguity set for problem \eqref{prob:dist_robust}. 

\subsection{Bootstrap-based Ambiguity Sets} Connecting back to the optimization problem, we construct an ambiguity set based on the bootstrap samples. Let $\Phi^*_k(\hat F_n)$ be the empirical distribution function of the $k$ bootstrap samples $\hat \beta_n^{*i}$, $i = 1, \dots, k$. 

For $q \in [1,\infty)$, the $q$-Wasserstein distance  between two probability distributions $\mu, \nu \in \mathcal{M}(\mathbf{R}^p)$ is defined as
\begin{equation}
\label{eq:wasserstein}
d_q(\mu, \nu) := \inf_{\gamma \in \mathcal{C}(\mu, \nu)} \left( \mathbf{E} \| \xi_1 - \xi_2 \| ^q \right)^{1/q}
\end{equation}
where the infimum is over joint distributions $\mathcal{C}(\mu, \nu)$ of $\xi_1$ and $\xi_2$ with marginals $\mu$ and $\nu$, respectively. The norm can be any norm on $\mathbf{R}^p$, but our bounds utilize the 2-norm. The $q-$Wasserstein distance defined in \eqref{eq:wasserstein} is a metric on the space of probability distributions with bounded $q$-th moment. 

We define a Wasserstein distance-based bootstrap ambiguity set
\begin{equation}
\label{eq:ambiguity_set}
\mathcal{P} = \{ \mu \in \mathcal{M}(\mathbf{R}^p) \mid d_q( \Phi^*_k(\hat F_n),\mu) \leq \epsilon \},
\end{equation}
i.e., the set of distributions on $\mathbf{R}^p$ within $\epsilon$ distance from the empirical bootstrap distribution $\Phi^*_k(\hat F_n)$. This ambiguity set is built on a finite amount $n$ of regression data and a finite number $k$ of bootstrap samples. Is it possible to select the radius $\epsilon$ to ensure that the ambiguity set contains the exact distribution $\Phi(F)$ with high probability? This would provide a finite-data out-of-sample safety guarantee for Problem \ref{prob:uncertain}. 


\section{Finite-Sample Bounds for Wasserstein Distances in Bootstrap Regression}
Our main result establishes a radius $\epsilon$ that ensures that the ambiguity set contains the exact distribution $\Phi(F)$ with high probability. This ambiguity set depends on the amount of data $n$, number of boostrap samples $k$, parameter dimension $p$, confidence level $\delta$, and residual tail properties.

\begin{thm}[Finite-Sample Bound for Wasserstein Distance in Bootstrap Regression] \label{thm:mainconcbound}
Suppose for any $n,k,p$ and confidence level $\delta \in (0,1)$ the Wasserstein radius is selected to be at least $\epsilon = \epsilon_1 + \epsilon_2 + \epsilon_3$, where 
$$\epsilon_1 = 
\begin{cases}
\left( \frac{\ln (3 c_1 \delta^{-1})}{c_2 k} \right)^{\frac{1}{\max(p,2)}} & \text{if } k \geq \frac{\ln (3 c_1 \delta^{-1})}{c_2} \\ 
\left( \frac{\ln (3 c_1 \delta^{-1})}{c_2 k} \right)^{\frac{1}{\alpha}} & \mbox{if }k < \frac{\ln (3 c_1 \delta^{-1})}{c_2} 
\end{cases}
$$ 
$$\epsilon_2 =  \sqrt{n}L \left( \frac{\bar L \| \varepsilon_{n1} \|_{\Psi_\alpha} \ln (6 \delta^{-1})}{c} \right)^{\frac{1}{\alpha}} +  \sqrt{n}L \sigma \sqrt{\frac{p+1}{n}} $$
$$\epsilon_3 =  \sqrt{n}L \begin{cases}
 \left( \frac{\ln (3 c_3 \delta^{-1})}{c_4 n} \right) & \text{if } n \geq \frac{\ln (3 c_3 \delta^{-1})}{c_4} \\ 
\left( \frac{\ln (3 c_3 \delta^{-1})}{c_3 n} \right)^{\frac{2}{\alpha}} & \mbox{if }n < \frac{\ln (3 c_3 \delta^{-1})}{c_4},
\end{cases} $$
with $L=\|(X_n^T X_n)^{-1}X_n^T \|$ and $\bar L = \frac{1}{\sqrt{n}}\| (\frac{1}{n} \mathbf{1} \mathbf{1}^T + \Pi)^{1/2} \|$ (here $\| \cdot \|$ denotes spectral norm), where $c_1, c_2, c_3, c_4, c$ are constants that depend on $\alpha$ and $p$.
Then we have 
$$ P( d_1(\Phi(F), \Phi_k^*(\hat F_n))  \geq \epsilon) \leq \delta .$$
\end{thm}


The proof is provided in the Appendix through a series of equalities/inequalities, based on \cite{bickel1981some, freedman1981bootstrapping, fournier2015rate}.

\textbf{Remark:} A standard assumption for consistency of bootstrapping for regression (see equation (1.4) in \cite{freedman1981bootstrapping}) implies that $L$ scales as $1/\sqrt n$. Since $\bar L$ also scales as $1/\sqrt n$, it can then be seen that for fixed confidence $\delta$ and residual tail properties, the radius $\eps$ of the ambiguity set shrinks to zero as $n$ and $k$ increase. Note that the rate of decrease becomes very slow with increasing parameter dimension $p$; this is an unavoidable curse of dimensionality for convergence of the empirical distribution in Wasserstein distance \cite{fournier2015rate}. 

While this result provides strong theoretical justification for the proposed approach, there are practical limitations in choosing $\eps$ according to the above result. First, the constants $c_1, c_2, c_3, c_4, c$ are not easy to compute analytically. Second, the upper bound on $\eps$ can be much larger than necessary. In practice, the radius is better viewed as a tuning parameter used to balance performance with risk of constraint violation.

\section{Finite-Sample Safety Guarantee and Exact Tractable Reformulation}
The finite-sample concentration bound of Theorem \ref{thm:mainconcbound} leads directly to the following out-of-sample safety guarantee.
\begin{cor}[Out-of-Sample Safety Guarantee] \label{cor:safety}
Consider the distributionlly robust optimization problem \eqref{prob:dist_robust} with bootstrap-based ambiguity set \eqref{eq:ambiguity_set}. Suppose the Wasserstein radius is selected according to Theorem \ref{thm:mainconcbound}. If \eqref{prob:dist_robust} is feasible, then for an optimizer $\hat x_{n,k}^*$ we have for $j=1,...,m,$
\begin{equation}
    P( \rho(a_j^T(\hat \beta_n)\hat x_{n,k}^* - b_j(\hat \beta_n)) \leq \Delta_j ) \geq 1 - \delta.
\end{equation}
\end{cor}
Corollary \ref{cor:safety} states that the risk constraint holds with high probability as long as the ambiguity set radius is selected appropriately, despite lack of knowledge of the exact distribution of $\hat \beta_n$ and the fact that we resample only from an approximate bootstrap empirical distribution. We can specialize Corollary \ref{cor:safety} to obtain an explicit constraint satisfaction guarantee for the original problem \eqref{prob:original}.
\begin{thm}
Consider the distributionlly robust optimization problem \eqref{prob:dist_robust} with bootstrap-based ambiguity set \eqref{eq:ambiguity_set} for the expected value risk meaasure $\rho(\cdot) = \mathbf{E}(\cdot)$ with risk limit $\Delta_j = 0$. Suppose the Wasserstein radius is selected according to Theorem \ref{thm:mainconcbound}. If \eqref{prob:dist_robust} is feasible, then for an optimizer $\hat x_{n,k}^*$ we have for $j=1,...,m,$
\begin{equation}
    P( a_j^T(\beta ) \hat x_{n,k}^* - b_j(\beta)) \leq 0 ) \geq 1 - \delta.
\end{equation}
\end{thm}
\begin{proof}
Since the constraints are linear in $\beta$ and the least-squares estimate $\hat \beta_n$ is unbiased, the result follows directly from Corollary \ref{cor:safety}.
\end{proof}

Though we now have an out-of-sample guarantee, the distributionally robust constraint in problem \eqref{prob:dist_robust} is infinite-dimensional. Fortunately, for certain risk measures the recent results of \cite{esfahani2018data} can be directly applied to obtain an exact reformulation as a finite-dimensional convex program. For concreteness, we state the result for the conditional value at risk (CVaR) at level $\gamma$, an important special case of the risk measure defined for a random variable $z$ by $$\mathbf{CVaR}_\gamma(z) := \inf_\tau \mathbf{E}\frac{1}{\gamma}\max(z - \tau, 0) + \tau. $$ In this case, the risk measure-constraint function composition is piecewise affine in $\hat \beta_n$. Exact reformulations are also available for more general risk measure-constraint function compositions, see \cite{esfahani2018data}.
\begin{thm}[Exact Tractable Reformulation, Corollary 5.1 of \cite{esfahani2018data}]
The distributionlly robust optimization problem \eqref{prob:dist_robust} with conditional value at risk $\rho(\cdot) = \mathbf{CVaR}_\gamma(\cdot)$ and bootstrap-based ambiguity set \eqref{eq:ambiguity_set} has an exact reformulation as the finite-dimensional convex optimization problem:
\begin{equation} \nonumber
\begin{aligned} 
 &\underset{x,s_j,\tau_j,\lambda_j}{\text{minimize}} \quad  && f(x)   \\ 
 &\text{subject to} && \epsilon \lambda_j + \frac{1}{k} \mathbf{1}^T s_j \leq \Delta_j \\ 
& &&  a_j^T( \hat \beta^{*i}_n)x - b_j( \hat \beta^{*i}_n) - (1 - \gamma) \tau_j \leq s_{ji}\\
& && \gamma \tau_j \leq s_j, \quad i=1,...,k, \quad j=1,...,m \\
& && \|x \|_{*} \leq \lambda_j \\
& && x \geq 0,
\end{aligned}
\end{equation}
with decision variables $x \in \mathbf{R}^d, s_j \in \mathbf{R}^k, \tau_j \in \mathbf{R}, \lambda_j \in \mathbf{R}$ for $j=1,..,m$, where $\|\cdot \|_*$ is the dual norm of that used in the Wasserstein distance definition \eqref{eq:wasserstein}. For the $1-$ and $\infty-$norms this is a linear program, and for the 2-norm it is a second-order cone program. 
\end{thm}

\section{Numerical Experiments}
To illustrate our approach, in this section we present a simple numerical example. 

For the regression layer, we set $n=100$, $p=10$, with entries of the data matrix $X_n$ drawn iid from a standard normal distribution and the residuals drawn iid from a uniform distribution on $[-1,1]$ (and thereafter fixed). The unknown parameter is set to $\beta_i = 0.1i$, $i=1,...,p$.

For the optimization layer, we aim to solve
\begin{equation}
\begin{aligned} 
 &\text{maximize} \quad  && \mathbf{1}^T x   \\ 
 &\text{subject to} && \beta^T x \leq 1 \\
& && x \geq 0.
\end{aligned}
\end{equation}
If $\beta$ were known, the optimal solution would be simply $x^* = \frac{1}{\beta_{n1}}[1 \ 0 \ \cdots 0]^T$. Note that the constraint $\beta^T x^* \leq 1$ is active in this case. However, since $\beta$ is unknown and must be estimated from noisy data, the certainty equivalent solution inevitably incurs a risk of constraint violation. In fact, the constraint violation can be arbitrarily large when the noise causes any parameter estimate component to be close to zero. In particular, for the estimate $\hat \beta_n$, an optimal solution for the certainty equivalent interface is $x_{CE}^* = \frac{1}{\hat \beta_{ni^*}}e_{i^*}$, where $i^* \in \arg \min_i \hat \beta_{ni}$. If $\hat \beta_{ni^*}$ becomes close to zero, then $\beta^T x_{CE}^*$ becomes large, resulting in a severe constraint violation. This motivates the proposed distributionally robust bootstrap optimization method, which offers a principled approach to limit out-of-sample constraint violation risk.

Leveraging the results of \cite{esfahani2018data}, we can formulate a tractable distributionally robust bootstrap optimization problem. We replace the inequality constraint $\beta^T x \leq 1$ with a limit on the conditional value at risk (CVaR) at level $\gamma$ given by $\inf_{\tau} \mathbf{E} \frac{1}{\gamma}\max(\hat \beta_n x - 1 - \tau,  0 ) + \tau \leq \Delta_1$. Then problem \eqref{prob:dist_robust} with bootstrap-based ambiguity set \eqref{eq:ambiguity_set} is equivalent to the finite-dimensional convex program
\begin{equation} \label{prob:opt_ex}
\begin{aligned} 
 &\underset{x,s,\tau,\lambda}{\text{maximize}} \quad  && \mathbf{1}^T x   \\ 
 &\text{subject to} && \varepsilon \lambda + \frac{1}{k} \mathbf{1}^T s \leq \Delta_1 \\ 
& &&  x^T \hat \beta^{*i}_n - (1 - \gamma) \tau \leq s_i + 1, \quad i=1,...,k \\
& && \gamma \tau \leq s \\
& && \|x \|_{2} \leq \lambda \\
& && x \geq 0
\end{aligned}
\end{equation}

\begin{figure}
    \centering
    \includegraphics[width=0.99\linewidth]{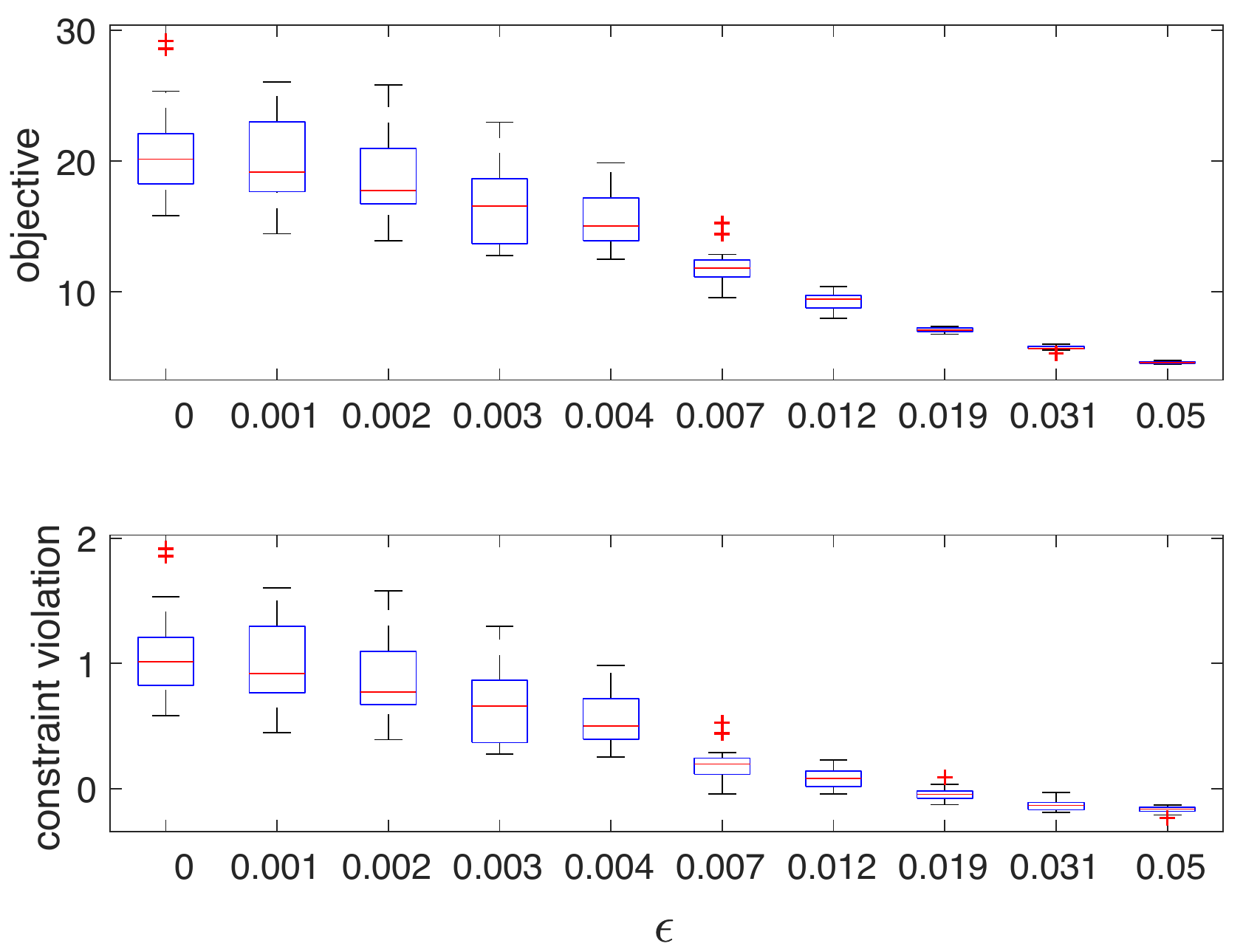}
    \caption{Performance-robustness trade-off with Wasserstein radius. }
    \label{fig:tradeoff}
\end{figure}

We consider a scenario where a noise realization has resulted in $\hat \beta_{ni^*} = 0.006$, in which case the certainty-equivalent solution gives a large constraint violation (where $\beta^T x^*_{CE} - 1$ exceeds 150). With this realization fixed, we solved $\eqref{prob:opt_ex}$ with $\Delta_1 = 0.1$, $\gamma = 0.05$, $k=30$ with various Wasserstein radii from $\eps = 0$ to $\eps = 0.05$. For each $\eps$ the bootstrap sampling was repeated 20 times. Boxplots for each radius $\eps$ are shown in Figure \ref{fig:tradeoff}. An inherent performance-robustness trade-off is clearly seen: increasing the Wasserstein radius  reduces the objective value but also reduces the amount of constraint violation $\beta^T \hat x_{n,k}^* - 1$. It can also be seen that for large enough radii, the amount of constraint violation is quite small despite the substantial error in the nominal least squares estimate. 


\section{Conclusions}
We considered a data-decision architecture that couples a statistical parameter estimation with a constraint optimization layer. To robustify the interface between layers, we proposed a distributionally robust bootstrap optimization algorithm based on Wasserstein ambiguity sets built directly on the bootstrap samples. Even though it is impossible to sample directly from the exact distribution of the parameter estimate, we showed that it is nevertheless possible to obtain a finite-data out-of-sample safety guarantee by suitably selecting the ambiguity set radius. Furthermore, the resulting distributionally robust optimization problem has an exact reformulation as a finite-dimensional convex program. 


\bibliographystyle{plain}
\bibliography{bibliography.bib}

\section*{Appendix}

Here, we prove the main result, Theorem \ref{thm:mainconcbound}. We first summarize some notation. We utilize the 2-norm in the Wasserstein distance \eqref{eq:wasserstein}. For matrices, $\| \cdot \|$ indicates the spectral norm. We define $F_n$ as the empirical distribution of the true unobserved residual $\ve_n$. The distribution of $\hat{\beta}_n$ based on $F_n$ is denoted by $\Phi(F_n)$. In summary, we have the following distributions in the residual and parameter spaces. 
\begin{center}
\begin{tabular}{ | l | l| l |} 
\hline
distributions & $\ve$ & $\beta$ \\ 
\hline
true residuals $\ve$ & $F$ & $\Phi(F)$ \\ 
unobserved residuals $\ve_n$  & $F_n$ & $\Phi(F_n)$ \\ 
observed residuals $\hat{\ve}_n$  & $\hat{F}_n$ & $\Phi(\hat{F}_n)$ \\ 
bootstrap empirical $\hat \beta^*_k$  & - & $\Phi^*_k(\hat F_n)$ \\ 
\hline
\end{tabular}
\end{center}

We also have the following useful convex concentration bound for $\alpha$-sub-exponential random variables \cite{sambale2020some}.
\begin{thm}[Proposition 1.2 in \cite{sambale2020some}] \label{convexconc}
Let $X_1,...,X_n$ be iid random variables with $\alpha$-sub-exponential tails for any $\alpha > 0$. Let $f: \mathbf{R}^n \rightarrow \mathbf{R}$ be convex and L-Lipschitz. Then for any $t\geq 0$ and some constant $c$
\begin{equation}
    P(| f(X) - \mathbf{E} f(X) | > t ) \leq 2 \exp \left(-\frac{ct^\alpha}{L^\alpha \| X_1 \|_{\Psi_\alpha}^\alpha}\right).
\end{equation}
\end{thm}
Now towards proving Theorem \ref{thm:mainconcbound}, from the triangle inequality, we have
\begin{equation} \label{eq:tri1}
    d_1(\Phi(F), \Phi_k^*(\hat F_n)) \leq d_1(\Phi(F), \Phi(\hat F_n)) + d_1(\Phi(\hat F_n), \Phi_k^*(\hat F_n))
\end{equation} 
The second term on the right is the distance between the exact distribution of the bootstrap estimate $\hat \beta^*_n$ and an empirical distribution function on $k$ independent bootstrap samples. The following shows that the norm of the bootstrap estimate $\| \hat \beta^*_n \|$ has the same tail property as the components of the residual vector quantified by Assumption \ref{assm:noise_dist_asubexp}.
\begin{lem} \label{betabound}
Consider the random vector $\hat \beta^*_n \sim \Phi(\hat F_n)$. Under Assumption \ref{assm:noise_dist_asubexp}, there is a constant $\eta$ such that $\mathbf{E}\exp(\eta \| \hat \beta^*_n \|^\alpha ) < \infty$.
\end{lem}
\begin{proof}
By definition, we have
\begin{equation}
    \hat \beta^*_n = (X_n^T X_n)^{-1}X_n^T y^*_n,
\end{equation}
where $y^*_n = X_n \hat \beta_n + \ve_n^*$ and $\ve_n^*$ has components iid from $\hat F_n$ (i.e., resampled with replacement from the centered residuals $\hat \ve_n$). Substituting $y^*_n$ and $\hat \beta_n$ gives
$$ \hat \beta^*_n = \beta + (X_n^T X_n)^{-1}X_n^T(\ve_n + \ve_n^*).$$
Thus, $\hat \beta^*_n$ is an affine transformation of $\ve_n + \ve_n^*$, a random vector with iid $\alpha$-sub-exponential components ($\ve_n$ iid from $F$, $\ve_n^*$ iid from $\hat F_n$) and therefore also has $\alpha$-sub-exponential components. The function $\| \hat \beta^*_n \|$ is convex in $\ve_n + \ve_n^*$, since it is the composition of a norm with an affine function, and it has Lipschitz constant $\|(X_n^T X_n)^{-1}X_n^T \|$. The result then follows from Theorem \ref{convexconc}, noting that centering does not affect the tail properties.
\end{proof}

Lemma \ref{betabound} allows to apply finite sample bounds from \cite{fournier2015rate}.
\begin{thm} [Wasserstein concentration for bootstrap samples \cite{fournier2015rate}] \label{thm:wassconc1}
Consider the distribution function $\Phi(\hat F_n) \in \mathcal{M}(\mathbf{R}^p)$ and corresponding empirical distribution function $\Phi_k^*(\hat F_n) \in \mathcal{M}(\mathbf{R}^p)$. Under Assumption \ref{assm:noise_dist_asubexp}, for some positive constants $c_1, c_2$ we have 
\begin{equation}
    P ( d_1(\Phi(\hat F_n), \Phi_k^*(\hat F_n)) \geq t ) \leq \begin{cases} c_1 \exp(-c_2 k t^{\max(p,2)} ),  t \leq 1 \\ c_1 \exp(-c_2 k t^\alpha),  t > 1 \end{cases}
\end{equation}
\end{thm}

Now we turn to bounding $d_1(\Phi_n(F), \Phi_n(\hat{F}_n))$.
\begin{lem}
For the Wasserstein distance induced by the Euclidean norm, we have
\begin{equation} \label{eq:bound1}
  d_1(\Phi(F), \Phi(\hat{F}_n)) \leq \sqrt n \| (X_n^T X_n)^{-1} X_n^T \| d_2(F,\hat F_n)
\end{equation}
\end{lem}
\begin{proof}
First, let $U$ and $V$ be random variables in $\mathbf{R}^d$ with iid components, and consider the linear transformations $A U$ and $A V$ for a matrix $A$. We abuse notation by writing using the random variables in the arguments of the Wasserstein distance as opposed to their corresponding distribution functions. From \cite{bickel1981some} Section 8 equation (8.3) we have for any $q \geq 1$
\begin{equation}
    d_q(A U, A V) \leq \| A \| d_q(U, V).
\end{equation}
Furthermore, for $q=1$ with the Euclidean norm, we have
\begin{align}
 d_1(U, V) &\leq \mathbf{E}[\sum_{i=1}^n (u_i - v_i)^2]^{1/2} \\
    &\leq [\sum_{i=1}^n \mathbf{E} (u_i - v_i)^2]^{1/2} \\
    &\leq [n \mathbf{E} (u_i - v_i)^2]^{1/2} = \sqrt n d_2(u_i,v_i),
\end{align}
where the second line follows from Jensen's inequality and linearity of expectation.
Putting the above together for $d_1(\Phi(F), \Phi(\hat{F}_n))$ yields the result.
\end{proof}

We now need a bound for $d_2(F,\hat F_n)$. Using the triangle inequality, we have
\begin{equation} \label{eq:tri2}
    d_2(F,\hat F_n) \leq d_2(F, F_n) + d_2(F_n,\hat F_n).
\end{equation}
The first term on the RHS is the distance between an exact distribution and empirical distribution, so the results of \cite{fournier2015rate} can again be applied to obtain a finite sample bound. In particular, we have
\begin{thm}[Wasserstein concentration for residual components \cite{fournier2015rate}] \label{thm:wassconc2}
Consider the distribution function $F \in \mathcal{M}(\mathbf{R})$ and the corresponding empirical distribution function on the unobserved residuals $F_n \in \mathcal{M}(\mathbf{R})$. Under Assumption \ref{assm:noise_dist_asubexp}, for some positive constants $c_3, c_4$ we have 
\begin{equation}
    P( d_2(F, F_n)^2 \geq t) \leq  \begin{cases} c_3 \exp(-c_4 n t^2 ) \quad t \leq 1 \\ c_3 \exp(-c_4 n t^{\alpha}) \quad t > 1, \end{cases}
\end{equation}
or equivalently,
\begin{equation}
    P( d_2(F, F_n) \geq t) \leq  \begin{cases} c_3 \exp(-c_4 n t ) \quad t \leq 1 \\ c_3 \exp(-c_4 n t^{\alpha/2}) \quad t > 1. \end{cases}
\end{equation}
\end{thm}

Finally, we turn to bounding $d_2(F_n,\hat F_n)$. 
\begin{lem} \label{lem:dineq}
We have
$d_2(F_n, \hat F_n) \leq \frac{1}{\sqrt n} \| (\frac{1}{n} \mathbf{1} \mathbf{1}^T + \Pi)^{1/2} \varepsilon_n \|_2$, where $\Pi = X_n(X_n^TX_n)^{-1}X_n^T \in \mathbf{R}^{n \times n}$ is the projection matrix onto the column space of $X_n$.
\end{lem}
\begin{proof}
Using equations (2.3) and (2.4) in \cite{freedman1981bootstrapping} gives
\begin{equation}
    d_2(F_n, \hat F_n)^2 \leq \left(\frac{1}{n}\mathbf{1}^T \varepsilon_n\right)^2 + \frac{1}{n}\| \hat \varepsilon_n - \varepsilon_n \|_2^2.
\end{equation}
It is easily verified that $\hat{\ve}_n - \ve_n = \Pi \ve_n$, where $\Pi = X_n(X_n^TX_n)^{-1}X_n^T \in \mathbf{R}^{n \times n}$. Combining the terms and taking the square root yields the result.
\end{proof}

\begin{lem} \label{lem:dineq2}
We have $$\mathbf{E} d_2(F_n, \hat F_n) \leq \mathbf{E} \frac{1}{\sqrt n} \| (\frac{1}{n} \mathbf{1} \mathbf{1}^T + \Pi)^{1/2} \varepsilon_n \|_2 \leq \sigma \sqrt{\frac{p+1}{n}}.$$
\end{lem}
\begin{proof}
The first inequality follows from Lemma \ref{lem:dineq}. Next note 
\begin{align}
(\mathbf{E} \frac{1}{n} \| (\frac{1}{n} \mathbf{1} \mathbf{1}^T + \Pi)^{1/2} \varepsilon_n \|_2)^2 &\leq \mathbf{E} \frac{1}{n} \| (\frac{1}{n} \mathbf{1} \mathbf{1}^T + \Pi)^{1/2} \varepsilon_n \|_2^2  =  \mathbf{E} \frac{1}{n} \varepsilon_n^T (\frac{1}{n} \mathbf{1} \mathbf{1}^T + \Pi) \varepsilon_n =  \mathbf{tr}[(\frac{1}{n^2} \mathbf{1} \mathbf{1}^T + \frac{1}{n}\Pi) \mathbf{E} \varepsilon_n \varepsilon_n^T] \\ \nonumber 
&= \sigma^2\frac{p+1}{n}.
\end{align}
The first step uses Jensen's inequality, as $(\mathbf{E}z)^2\leq \mathbf{E}z^2$, since $(\cdot)^2$ is convex. Taking the square root gives the result.
\end{proof}

\begin{prop} \label{prop:conc1}
Under Assumption \ref{assm:noise_dist_asubexp}, for some constant $c$ we have
\begin{equation} 
    P\left( d_2(F_n, \hat F_n) \geq \sigma \sqrt{\frac{p+1}{n}} + t \right) \leq 2\exp\left(-\frac{c t^\alpha}{\bar L^\alpha \| \varepsilon_{n1} \|_{\Psi_\alpha}^\alpha }\right)
\end{equation}
where $\bar L = \frac{1}{\sqrt{n}}\| (\frac{1}{n} \mathbf{1} \mathbf{1}^T + \Pi)^{1/2} \| $.
\end{prop}
\begin{proof}
Since $\frac{1}{\sqrt n} \| (\frac{1}{n} \mathbf{1} \mathbf{1}^T + \Pi)^{1/2} \varepsilon_n \|_2$ is a convex and Lipschitz function of $\varepsilon_n$ with Lipschitz constant $L = \frac{1}{\sqrt{n}}\| (\frac{1}{n} \mathbf{1} \mathbf{1}^T + \Pi)^{1/2} \| $, we have from Theorem \ref{convexconc} that
\begin{align} \nonumber
    P\left(\frac{1}{\sqrt n} \| (\frac{1}{n} \mathbf{1} \mathbf{1}^T + \Pi)^{1/2} \varepsilon_n \|_2 \geq \mathbf{E}\frac{1}{\sqrt n} \| (\frac{1}{n} \mathbf{1} \mathbf{1}^T + \Pi)^{1/2} \varepsilon_n \|_2 + t \right) \leq 2\exp\left(-\frac{c t^\alpha}{L^\alpha \| \varepsilon_{n1} \|_{\Psi_\alpha}^\alpha }\right).
\end{align}
The result then follows from Lemmas \ref{lem:dineq} and \ref{lem:dineq2}.
\end{proof}

Now we are ready to put the bounds together. We will utilize the following Lemma.
\begin{lem}
\label{lem:additive_prob}
Consider the random variables $a, b, c: \Omega \rightarrow \R$. Suppose, for all $\omega \in \Omega$, $a(\omega) \leq b(\omega) + c(\omega)$,  Then, for any $\eps_a  = \eps_b + \eps_c$
\begin{align}
\label{eq:additive_prob}
P( a(\omega) \geq \eps_a) \leq P( b(\omega) \geq \eps_b) + P(c(\omega) \geq \eps_c).
\end{align} 
\end{lem}
\begin{proof}
For simplicity in notation, we drop the dependence on $\omega$ below. Since $a\leq b + c$ we have 
\begin{align}
\label{eq:part1}
P(a \geq \eps_b + \eps_c) \leq P( b + c \geq \eps_b + \eps_c)
\end{align}
Furthermore, the event $b+c \geq \eps_b + \eps_c$ implies the event $ b \geq \eps_b $ or the event $c \geq \eps_c$. Hence, 
\begin{align}
\label{eq:part2}
&P(b+c \geq \eps_b + \eps_c)\leq P(b \geq \eps_b) + P(c \geq \eps_c).
\end{align}
Putting  \eqref{eq:part1}, \eqref{eq:part2} we get the desired result. 
\end{proof}

\noindent \textit{Proof of Theorem \ref{thm:mainconcbound}}.
By now we have obtained concentration bounds for each term in the upper bounds from the inequalities \eqref{eq:tri1}, \eqref{eq:bound1}, and \eqref{eq:tri2} from Theorems \ref{thm:wassconc1} and \ref{thm:wassconc2} and Proposition \ref{prop:conc1}. Combining these inequalities and applying Lemma \ref{lem:additive_prob} twice yields
\begin{align} \nonumber
    P( d_1(\Phi(F), \Phi_k^*(\hat F_n))  \geq \epsilon) \leq P ( d_1(\Phi(\hat F_n), \Phi_k^*(\hat F_n)) \geq \epsilon_1)
    + P( \sqrt n L d_2(F_n, \hat F_n) \geq \epsilon_2 ) + P( \sqrt n L d_2(F, F_n) \geq \epsilon_3 )
\end{align}
with $\epsilon = \epsilon_1 + \epsilon_2 + \epsilon_3$. Equating the right-hand side to $\delta$, allocating $\delta/3$ to each term, and using the bounds from Theorems \ref{thm:wassconc1} and \ref{thm:wassconc2} and Proposition \ref{prop:conc1}, we can invert each bound to solve for the corresponding radii $\epsilon_1, \epsilon_2, \epsilon_3$, and the result is obtained.

\end{document}